\documentclass[12pt]{amsart}

\usepackage{amsmath,amssymb,amsthm}
\usepackage{bm}
\usepackage{mathrsfs}
\usepackage[all]{xy}
\usepackage{booktabs}
\usepackage{fullpage}
\usepackage[pagebackref=true]{hyperref}
\usepackage{mathtools}
\usepackage[abbrev,msc-links]{amsrefs}
\usepackage[]{todonotes}
\hypersetup{
 hidelinks,
 bookmarksopen=true,
}

\newcommand{\rl}{\mathbb{R}}
\newcommand{\cx}{\mathbb{C}}

\newcommand{\CE}{\mathcal{E}}

\newcommand{\CH}{\mathcal{H}}

\newcommand{\ai}{\sqrt{-1}}

\newcommand{\ddbar}{\partial \bar{\partial}}

\theoremstyle{plain}
\newtheorem{theorem}{Theorem}[section]

\newtheorem{lemma}[theorem]{Lemma}
\newtheorem{proposition}[theorem]{Proposition}
\newtheorem{corollary}[theorem]{Corollary}
\theoremstyle{definition}
\newtheorem{definition}[theorem]{Definition}

\theoremstyle{definition}
\newtheorem{remark}[theorem]{Remark}

\begin{document}

\title{Relative uniform $K$-stability over models implies existence of extremal metrics}
\author{Yoshinori Hashimoto}
\date{\today}
\address{Department of Mathematics, Osaka Metropolitan University, 3-3-138, Sugimoto, Sumiyoshi-ku, Osaka, 558-8585, Japan.}\email{yhashimoto@omu.ac.jp}

\begin{abstract}
	We prove that an extremal metric on a polarised smooth complex projective variety exists if it is $\mathbb{G}$-uniformly $K$-stable relative to the extremal torus over models, extending a result due to Chi Li \cite{Li20} for constant scalar curvature K\"ahler metrics.
\end{abstract}

\maketitle

\tableofcontents

\section{Introduction}

Let $(X,L)$ be a polarised smooth projective variety over $\mathbb{C}$ of complex dimension $n$. The existence of canonical K\"ahler metrics on $(X,L)$ has been studied intensively in recent decades, particularly in connection to the Yau--Tian--Donaldson conjecture, which states that $(X,L)$ admits a constant scalar curvature K\"ahler (cscK) metric if and only if it is $K$-polystable. Li \cite{Li20} made a significant progress towards this conjecture by proving that the existence of cscK metrics follows from the $\mathbb{G}$-uniform $K$-stability over models. Part of this result (uniform $K$-stability over $\mathcal{E}^{1 , \mathrm{NA}}$ implying cscK metrics) was generalised to arbitrary compact K\"ahler manifolds by Mesquita-Piccione \cite[Theorem A]{MP24}.

A similar conjecture for Calabi's extremal metrics \cite{Cal1} was proposed by Sz\'ekelyhidi \cite{Sze07}, which states that there exists an extremal metric on $(X,L)$ if and only if it is relatively $K$-polystable. It is natural to consider a generalisation of Li's result to the extremal case, which is the aim of this paper. The main result is the following.

\begin{theorem} \label{thexcl}
	Let $\mathbb{G}$ be the complexification of a maximal compact subgroup $\mathbb{K}$ of $\mathrm{Aut}_0 (X,L)$. If $(X,L)$ is $\mathbb{G}$-uniformly $K$-stable relative to the extremal torus over models, then $(X,L)$ admits a $\mathbb{K}$-invariant extremal metric.
\end{theorem}

See Definition \ref{dfguksrtm} for the details of the stability condition used in the theorem above. We prove this theorem by showing (Corollary \ref{cctexjena}) that the correction term for the modified non-Archimedean Mabuchi energy is continuous with respect to the strong topology in $(\mathcal{E}^{1, \mathrm{NA}})^{\mathbb{K}}$, and hence Li's argument applies without any significant change. We also note that the entire proof heavily depends on the non-Archimedean theory of Boucksom--Jonsson \cite{BJ18}.

\begin{remark}
As this paper was nearing completion, the author learned that Boucksom--Jonsson \cite[Theorem A']{BJ25} proved that the existence of $(v,w)$-weighted cscK metrics is equivalent to the $(v,w)$-uniform $K$-stability for models. Thus, while this work was done independently, Theorem \ref{thexcl} can be regarded as forming a small proper subset of their result, proving only one direction of \cite[Theorem A']{BJ25} for extremal metrics.
\end{remark}

Some of our results also overlap with Apostolov--Jubert--Lahdili \cite{AJL23}, Han--Li \cite{HanLi23}, Inoue \cite{Ino}, and Lahdili \cite{Lah19}, as pointed out in Remarks \ref{rmhlegdf} and \ref{rmhleaegdf}.

Concerning the other direction of the conjecture, i.e.~the extremal metric implying relative $K$-polystability (over test configurations), we recall the following well-known results. Mabuchi \cite{Mab14} proved that the existence of extremal metric implies the $K$-polystability of $(X,L)$ relative to the extremal torus, extending an earlier result due to Stoppa--Sz\'ekelyhidi \cite{StoSze} who proved the $K$-stability relative to the maximal torus. In section \ref{scrlkpst}, we remark that this result can be proved by slightly modifying the variational argument of Berman--Darvas--Lu \cite{BDL20}, which is likely well-known to the experts.

\medskip

\noindent \textbf{Organisation of the paper.} After reviewing the preliminaries in section \ref{scprlm}, the key new ingredients for Theorem \ref{thexcl} are proved in section \ref{scmtcrs}. Theorem \ref{thexcl} is proved in section \ref{scpfomnth}, building up on the proof for the cscK case by Li \cite{Li20}. Section \ref{scrlkpst} is a remark on the relative $K$-polystability of the extremal manifold relative to the extremal torus.

\medskip

\noindent \textbf{Acknowledgements} The author thanks Vestislav Apostolov, Thibaut Delcroix, Kento Fujita, Eiji Inoue, Mattias Jonsson, Julien Keller, Yan Li, and Yuji Odaka for helpful discussions, and Pietro Mesquita-Piccione for helpful discussions and pointing out using \cite[Corollary 6.7]{BBJ} in the proof of Lemma \ref{lmappgsg}. This work is partially supported by JSPS KAKENHI Grant Number JP23K03120, and was partly carried out when the author was staying at Montr\'eal as a CRM-Simons scholar during the Thematic Program in Geometric Analysis in 2024; he thanks the organisers, the CRM, and the Simons Foundation for the funding and hospitality.

\section{Preliminaries} \label{scprlm}

\subsection{Calabi's extremal metrics}

We first fix a maximal compact subgroup $\mathbb{K}$ of $\mathrm{Aut}_0 (X,L)$, where $\mathrm{Aut}_0 (X,L)$ is the identity component of the group consisting of automorphisms of $X$ which lift to the total space of $L$. Following the notation in \cite[section 2.1.3]{Li20}, we write
\begin{itemize}
	\item $\mathbb{G} = \mathbb{K}^{\mathbb{C}}$ for the complexification of $\mathbb{K}$,
	\item $\mathbb{T}$ for the identity component of the centre of $\mathbb{G}$.
\end{itemize}

We also fix a reference K\"ahler metric $\omega \in c_1 (L)$, which we assume is $\mathbb{K}$-invariant, and write $\CH$ for the space of K\"ahler potentials with respect to $\omega$, i.e.
\begin{equation*}
	\CH := \{ \varphi \in C^{\infty}\ (X , \rl) \mid \omega_{\varphi}:= \omega + \ai \ddbar \varphi >0 \}.
\end{equation*}
We write $\CE^1$ for the finite energy space, which is the completion of $\CH$ with respect to the distance $d_1$; see Darvas' monograph \cite{darmonograph} for more details. We write $(\mathcal{H})^{\mathbb{K}}$ for the $\mathbb{K}$-invariant K\"ahler potentials, and similarly for $(\mathcal{E}^1)^{\mathbb{K}}$.

Given a vector field $v$ induced by the $\mathbb{G}$-action and a K\"ahler metric $\omega_{\varphi}$ with $\varphi \in (\mathcal{H})^{\mathbb{K}}$, we define a smooth function $\theta (\varphi)$ (which is in general $\mathbb{C}$-valued) satisfying
\begin{equation*}
	\iota (v) \omega_{\varphi} = \ai \bar{\partial} \theta (\varphi),
\end{equation*}
which is well-defined up to an additive constant, and call it a \textbf{holomorphy potential} of $v$ with respect to $\omega_{\varphi}$.

A K\"ahler metric $\omega_{\varphi} \in (\mathcal{H})^{\mathbb{K}}$ is said to be an \textbf{extremal metric} if it satisfies
\begin{equation*}
	\bar{\partial} \mathrm{grad}^{1,0}_{\omega_{\varphi}} S(\omega_{\varphi}) = 0.
\end{equation*}
Futaki--Mabuchi \cite[Theorem C and Corollary D]{FutMab95} show that there exists a unique vector field called the \textbf{extremal vector field} $v_{\mathrm{ext}}$ which lies in the Lie algebra of $\mathbb{T}$ and agrees with the Hamiltonian Killing vector field of $S(\omega_{\varphi})$ when an extremal metric $\omega_{\varphi}$ exists. This vector field is periodic by \cite[Theorem F]{FutMab95} and \cite[Theorem 1.11]{Nak}, and generates a one-dimensional torus $\mathbb{T}_{\mathrm{ext}}$ in $\mathbb{T}$, which we call the \textbf{extremal torus}.

\subsection{Modified Mabuchi energy}

We write $V := \int_X c_1(L)^n$ for the volume of $(X,L)$, and $\bar{S}:= - n \int_X c_1(K_X)c_1(L)^{n-1} / V$ for the average scalar curvature. We recall standard functionals
\begin{align*}
	E(\varphi) &:= \frac{1}{V} \sum_{j=0}^n \int_X \varphi \omega^j_{\varphi} \wedge \omega^{n-j} , \\
	E_{\mathrm{Ric} (\omega)}(\varphi) &:= \frac{1}{V} \sum_{j=0}^{n-1} \int_X \varphi \omega^j_{\varphi} \wedge \omega^{n-j-1} \wedge \mathrm{Ric} (\omega ) , \\
	J (\varphi) &:=  \frac{1}{V} \int_X \varphi \omega^n - \frac{1}{n+1} E (\varphi ),
\end{align*}
and the entropy
\begin{equation*}
	H(\varphi) := \frac{1}{V} \int_X \log \left( \frac{\omega^n_{\varphi}}{\omega^n} \right) \omega^n_{\varphi} ,
\end{equation*}
defined for $\varphi \in \CH$. The Mabuchi energy $M : \mathcal{H} \to \mathbb{R}$ is defined by
\begin{equation*}
	M (\varphi) := \frac{\bar{S}}{n+1} E (\varphi) - E_{\mathrm{Ric}(\omega )} (\varphi ) +H(\varphi).
\end{equation*}

We define the functional $J_{\mathrm{ext}}$, following e.g.~\cite[section 4.1]{BB17}.
\begin{definition}
The functional $J_{\mathrm{ext}} : (\mathcal{H})^{\mathbb{K}} \to \mathbb{R}$ is defined by
	\begin{equation*}
	J_{\mathrm{ext}} (\varphi) := \frac{1}{V} \int_0^1  \int_X \dot{\varphi}_t \theta (\varphi_t) \frac{\omega_{\varphi_t}^n}{n!} dt
\end{equation*}
where $\{ \varphi_t \}_{0 \le t \le 1} \subset (\mathcal{H})^{\mathbb{K}}$ is any path connecting $\varphi_0 = 0$ and $\varphi_1 = \varphi$, and $\theta (\varphi_t)$ is the holomorphy potential of the extremal vector field with respect to the K\"ahler metric $\omega_{\varphi_t}$.
\end{definition}

We can show that the above definition is well-defined, that it does not depend on the path $\{ \varphi_t \}_{0 \le t \le 1} \subset (\mathcal{H})^{\mathbb{K}}$ and depends only on the endpoints (see e.g.~\cite[section 4.1]{BB17}). It is important that the domain is $(\mathcal{H})^{\mathbb{K}}$ (and not $\mathcal{H}$) so that $\theta (\varphi_t)$ is an $\mathbb{R}$-valued function.

\begin{definition}
	The \textbf{modified Mabuchi energy} $M_{\mathrm{ext}} : (\mathcal{H})^{\mathbb{K}} \to \mathbb{R}$ is defined by
\begin{equation*}
	M_{\mathrm{ext}} (\varphi) := M(\varphi)+ J_{\mathrm{ext}} (\varphi).
\end{equation*}
\end{definition}

It is well-known that $M_{\mathrm{ext}}$ is convex along $C^{1,\bar{1}}$-geodesics and its critical point is the extremal metric \cite{BB17}. He \cite[Proposition 2.2]{He19} proved that $J_{\mathrm{ext}}$ extends as a $d_1$-continuous function to $J_{\mathrm{ext}} : (\mathcal{E}^1)^{\mathbb{K}} \to \mathbb{R}$ which is affine linear on finite energy geodesics (based on the result of Berman--Berndtsson \cite[section 4]{BB17}). This implies in particular that the modified Mabuchi energy extends as a $d_1$-lower semicontinuous functional to $M_{\mathrm{ext}} : (\mathcal{E}^1)^{\mathbb{K}} \to \mathbb{R} \cup \{ + \infty \}$, and also proves that $M_{\mathrm{ext}}$ is convex along finite energy geodesics in $(\mathcal{E}^1)^{\mathbb{K}}$, as in \cite[Corollary 2.2]{He19}.

\subsection{Relative $K$-stability}

We recall here the bare minimum of materials concerning the $K$-stability, without reviewing various concepts that are necessary for its definition, since the details are involved. We follow the formulation of $K$-stability in terms of the non-Archimedean metrics as developed in \cite{BBJ,BHJ1,BHJ2,BJ18}, to which the reader is referred for more details and explanations. We write $\mathcal{H}^{\mathrm{NA}}$ for the set of all non-Archimedean metrics on $L$, which is a set of all equivalence classes of semiample test configurations for $(X,L)$, where the equivalence relation is given by the pullback.

The uniform $K$-stability can be defined in terms of the non-Archimedean Mabuchi energy $M^{\mathrm{NA}}$ and the non-Archimedean $J$-energy $J^{\mathrm{NA}}$; see \cite[section 7]{BHJ1} for the definitions. We recall that $(X,L)$ is said to be uniformly $K$-stable if there exists $\epsilon >0$ such that
\begin{equation*}
	M^{\mathrm{NA}} (\phi) \ge \epsilon J^{\mathrm{NA}} (\phi )
\end{equation*}
holds for any $\phi \in \mathcal{H}^{\mathrm{NA}}$.

When the automorphism group $\mathrm{Aut}_0 (X,L)$ is non-trivial, it is well-known that we need to consider a group equivariant version of the uniform $K$-stability. Firstly, we write $(\mathcal{H}^{\mathrm{NA}})^{\mathbb{K}}$ for the equivalence classes of $\mathbb{G}$-equivariant test configurations for $(X,L)$, following \cite[section 2.1.3]{Li20}. In place of $J^{\mathrm{NA}}$ used above, we use the reduced $J$-norm defined as
\begin{equation*}
		J^{\mathrm{NA}}_{\mathbb{T}} (\phi) := \inf_{\xi \in N_{\mathbb{R}}} J^{\mathrm{NA}} (\phi_{\xi})
\end{equation*}
for $\phi \in (\mathcal{H}^{\mathrm{NA}})^{\mathbb{K}}$, where $N_{\mathbb{R}} := \mathrm{Hom}_{\mathrm{Grp}} (\cx^*, \mathbb{T}) \otimes \mathbb{R}$ and $\phi_{\xi}$ is the twist of $\phi$ by $\xi$ (see \cite{His16} and \cite[Lemma 2.19]{Li20}).

We also need a modification term for the non-Archimedean Mabuchi energy when we deal with extremal metrics. Take $\phi \in (\mathcal{H}^{\mathrm{NA}})^{\mathbb{K}}$, represented by a test configuration $(\mathcal{X} , \mathcal{L})$. Following Yao \cite[Definition 3.1]{Yao19}, we define
\begin{equation} \label{eqyyaojnatc}
	J_{\mathrm{ext}}^{\mathrm{NA}} (\phi) := \frac{1}{V / n!} \frac{\mathcal{L}_{\beta}^{n+2}}{(n+2)!} - \frac{1}{(V / n!)^2} \left( \frac{\mathcal{L}^{n+1}}{(n+1)!} \right)^2,
\end{equation}
where $\mathcal{L}_{\beta}$ is defined as follows: $\mathcal{X}_{\beta}$ is the total space of the product test configuration of $(\mathcal{X} , \mathcal{L})$ with respect to the $\mathbb{C}^*$-action $\beta$ of the extremal vector field (so $\mathcal{X}_{\beta}$ is a product test configuration of a test configuration of $X$), and $\mathcal{L}_{\beta}$ is the corresponding $\mathbb{Q}$-Cartier divisor on $\mathcal{X}_{\beta}$, by noting that $(\mathcal{X} , \mathcal{L})$ is a test configuration whose defining $\mathbb{C}^*$-action commutes with the action by $\beta$. All the test configurations above are compactified over $\mathbb{P}^1$ as usual, and the intersection numbers above are computed with respect to this compactification. The above formula is well-defined, irrespectively of the representative $(\mathcal{X} , \mathcal{L})$ for $\phi$ chosen, by \cite[Proposition 3.2]{Yao19}. We also note that $b_0$ that appears in \cite[Definition 3.1]{Yao19} is exactly
\begin{equation*}
	b_0 = \lim_{s \to + \infty} \frac{E(\Phi(s))}{s} = E^{\mathrm{NA}} (\phi) = \frac{\mathcal{L}^{n+1}}{(n+1)!}
\end{equation*}
by \cite[Proposition 3]{Don05}, where $\{ \Phi (s) \}_{s \ge 0}$ is a $C^{1, \bar{1}}$-geodesic ray associated to $\phi$ (see \cite[Definition 3.3 (B)]{Yao19}). Note that $\{ \Phi (s) \}_{s \ge 0}$, as constructed by \cite[Theorem 1.1]{PS07} (see also \cite[Proposition 2.7]{Ber2016}), is a maximal geodesic ray in the sense of \cite[Definition 6.5]{BBJ} by \cite[Lemma 5.3]{BBJ}, since its construction shows that it has algebraic singularities \cite[Definition 4.5]{BBJ}; in fact it is known that there is a one-to-one correspondence between $\CH^{\mathrm{NA}}$ and geodesic rays with algebraic singularities, as pointed out in \cite[page 608 and section 4.4]{BBJ}.

Identifying the test configuration $(\mathcal{X} , \mathcal{L})$ and the $\mathbb{C}^*$-action $\alpha$ which defines it, the notation $\langle \alpha , \beta \rangle$ is also used \cite{Sze07} and can be regarded as a generalisation of the Futaki--Mabuchi bilinear form \cite{FutMab95} to non-product test configurations. It is the result of Yao \cite[Theorem 3.9]{Yao19} that $J_{\mathrm{ext}}^{\mathrm{NA}} (\phi)$ as defined above agrees with $\langle \alpha , \beta \rangle$, and we have
\begin{equation*}
	\lim_{s \to + \infty} \frac{J_{\mathrm{ext}} (\Phi (s))}{s} = J_{\mathrm{ext}}^{\mathrm{NA}} (\phi)
\end{equation*}
which is also due to Yao \cite[Theorem 3.7]{Yao19}.

\begin{remark} \label{rmhlegdf}
Han--Li \cite[Lemma 5.2, Proposition 5.8]{HanLi23} also proved similar results (in their notation, $J_{\mathrm{ext}}$ is written as $\bm{E}_g$ for an appropriate choice of $g$).	
\end{remark}

With all these understood and following \cite{Sze07}, the version of $K$-stability that we use in this paper can be defined as follows.

\begin{definition}
	$(X,L)$ is said to be \textbf{$\mathbb{G}$-uniformly $K$-stable relative to $\mathbb{T}_{\mathrm{ext}}$} if there exists $\epsilon >0$ such that
	\begin{equation*}
		M^{\mathrm{NA}} (\phi) +  J_{\mathrm{ext}}^{\mathrm{NA}} (\phi) \ge \epsilon J^{\mathrm{NA}}_{\mathbb{T}} (\phi)
	\end{equation*}
	holds for any $\phi \in (\mathcal{H}^{\mathrm{NA}})^{\mathbb{K}}$.
\end{definition}

We can consider stability relative to a torus larger than $\mathbb{T}_{\mathrm{ext}}$; in that case the invariant $J_{\mathrm{ext}}^{\mathrm{NA}}$ needs to change to eliminate the contributions from the larger torus.

The space of non-Archimedean metrics $\mathcal{H}^{\mathrm{NA}}$ admits a completion to the space $\mathcal{E}^{1 , \mathrm{NA}}$, analogously to the relationship between $\mathcal{H}$ and $\mathcal{E}^1$. The space $\mathcal{E}^{1 , \mathrm{NA}}$ can be endowed with a topology called the strong topology \cite[section 12.1]{BJ18}, and we can also define $(\mathcal{E}^{1 , \mathrm{NA}})^{\mathbb{K}}$ to be the set of elements in $\mathcal{E}^{1 , \mathrm{NA}}$ which can be realised as a limit of a decreasing sequence in $(\mathcal{H}^{\mathrm{NA}})^{\mathbb{K}}$ \cite[section 2.1.3]{Li20}. While $\mathcal{E}^{1 , \mathrm{NA}}$ and $(\mathcal{E}^{1 , \mathrm{NA}})^{\mathbb{K}}$ play a very important role for us, detailed explanation of their foundational properties are out of reach of this paper. The reader is referred to \cite{BJ18,BBJ,Li20} for more details.

There is another generalisation of test configurations, called models \cite{Li20}. The definition of models is similar to that of test configurations, but we are allowed to consider non-semiample polarisations for models \cite[Definition 2.1]{Li20}. Models define filtrations of the section ring of $(X,L)$, called model filtrations \cite[Definition 2.7]{Li20}, and hence a plurisubharmonic function on $X^{\mathrm{an}}$ in the sense introduced in \cite{BJ18}. Following \cite[Definition 2.7]{Li20}, we write $\mathrm{PSH}^{\mathfrak{M} , \mathrm{NA}}$ for the model filtrations. We have
\begin{equation*}
	\mathcal{H}^{\mathrm{NA}} \subset \mathrm{PSH}^{\mathfrak{M} , \mathrm{NA}} \subset \mathcal{E}^{1 , \mathrm{NA}} ,
\end{equation*}
by noting that a plurisubharmonic function on $X^{\mathrm{an}}$ corresponding to a filtration is an increasing limit of its canonical approximants \cite[Definition 1.12]{BJK1}. The group equivariant version $(\mathrm{PSH}^{\mathfrak{M} , \mathrm{NA}})^{\mathbb{K}}$ is defined as $\mathrm{PSH}^{\mathfrak{M} , \mathrm{NA}} \cap (\mathcal{E}^{1 , \mathrm{NA}})^{\mathbb{K}}$, following \cite[Definition 2.25]{Li20}.

\section{Main technical results} \label{scmtcrs}

We assume that all geodesics in this paper emanate from $\Phi_{\mathrm{ref}} :=0 \in (\mathcal{H})^{\mathbb{K}}$. We also note that we have $E (\Phi_{\mathrm{ref}}) = J_{\mathrm{ext}} (\Phi_{\mathrm{ref}}) = 0$ in our normalisation of these functionals. We start with the following lemma which is likely well-known to the experts.

\begin{lemma} \label{lmappgsg}
	Let $\{ \phi_j \}_j \subset \mathcal{E}^{1 , \mathrm{NA}}$ be a decreasing net of non-Archimedean metrics converging to $\phi \in \mathcal{E}^{1 , \mathrm{NA}}$. Let $\{ \Phi (s) \}_{s \ge 0} \subset \mathcal{E}^1$  (resp.~$\{ \Phi_j (s) \}_{s \ge 0} \subset \mathcal{E}^1$) be the (unique) maximal geodesic ray associated to $\phi \in \mathcal{E}^{1 , \mathrm{NA}}$ (resp.~$\phi_j \in \mathcal{E}^{1 , \mathrm{NA}}$), which exists by \cite[Theorem 6.6]{BBJ}. Then
	\begin{equation*}
		\lim_j d_1 (\Phi_j (1), \Phi (1)) = 0 .
	\end{equation*}
\end{lemma}

\begin{proof}
	Note first that we have
	\begin{equation*}
		\lim_{j} E^{\mathrm{NA}} (\phi_j) = E^{\mathrm{NA}} (\phi),
	\end{equation*}
	as $\phi_j$ decreases to $\phi$. Since $\Phi$ and $\Phi_j$ are maximal, we find $\lim_{s \to + \infty} E (\Phi (s) )/s = E^{\mathrm{NA}} (\phi)$ and also $\lim_{s \to + \infty} E (\Phi_j (s) )/s = E^{\mathrm{NA}} (\phi_j)$. Since $E^{\mathrm{NA}} (\phi_j) \to E^{\mathrm{NA}} (\phi)$ by assumption, we have
	\begin{equation*}
		\lim_{j} \lim_{s \to + \infty}\frac{1}{s} \left( E(\Phi_j (s)) - E (\Phi (s)) \right)=0.
	\end{equation*}

	Note that we have $\Phi_j \ge \Phi$ since $\phi_j$ decreases to $\phi$ and that $\Phi_j$, $\Phi$ are both maximal (see e.g.~\cite[Definition 6.5]{BBJ}). Then a result by Darvas \cite[Proof of Corollary 4.14]{Darvas15} shows that $E(\Phi_j (s)) - E (\Phi (s)) = d_1 (\Phi_j (s) , \Phi (s))$ since $\Phi_j \ge \Phi$. Since $d_1 (\Phi_j (s) , \Phi (s))$ is convex in $s$ by \cite[Proposition 5.1]{BDL17}, the difference quotient $d_1 (\Phi_j (s) , \Phi (s))/s$ is monotonically increasing in $s$, which in turn implies that we have
	\begin{align*}
		0 &= \lim_{j} \lim_{s \to + \infty}\frac{1}{s} \left( E(\Phi_j (s)) - E (\Phi (s)) \right) \\
		&\ge \lim_{j} d_1 (\Phi_j (1) , \Phi (1)) \ge 0,
	\end{align*}
	hence the result. Note that we can also get the same result by using \cite[Corollary 6.7]{BBJ}.
\end{proof}

\begin{lemma} \label{lmapenatc}
	Let $\{ \phi_j \}_j \subset (\mathcal{E}^{1 , \mathrm{NA}})^{\mathbb{K}}$ be a decreasing net of non-Archimedean metrics converging to $\phi \in (\mathcal{E}^{1 , \mathrm{NA}})^{\mathbb{K}}$. Let $\{ \Phi (s) \}_{s \ge 0} \subset (\mathcal{E}^1)^{\mathbb{K}}$ (resp.~$\{ \Phi_j (s) \}_{s \ge 0} \subset (\mathcal{E}^1)^{\mathbb{K}}$) be the maximal geodesic ray associated to $\phi \in (\mathcal{E}^{1 , \mathrm{NA}})^{\mathbb{K}}$ (resp.~$\phi_j \in (\mathcal{E}^{1 , \mathrm{NA}})^{\mathbb{K}}$). Then
	\begin{equation*}
		\lim_{s \to + \infty} \frac{J_{\mathrm{ext}} (\Phi (s))}{s} = \lim_{j} \lim_{s \to + \infty} \frac{J_{\mathrm{ext}} (\Phi_j (s))}{s}.
	\end{equation*}
\end{lemma}

We first establish the following claim, which seems to have a folklore status among the experts. Its proof is also embedded in \cite[Proof of Theorem 8.6]{FH24}, but we provide an alternative proof here.
\begin{lemma} \label{lmmgdrkivt}
For any $\phi \in (\mathcal{E}^{1 , \mathrm{NA}})^{\mathbb{K}}$ we can find a maximal geodesic ray $\{ \Phi (s) \}_{s \ge 0} \subset (\mathcal{E}^1)^{\mathbb{K}}$ associated to it.
\end{lemma}

\begin{proof}
First note that Berman--Boucksom--Jonsson \cite[Theorem 6.6]{BBJ} prove that there indeed exists a (unique) maximal geodesic ray in $\{ \Phi_0 (s) \}_{s \ge 0} \subset \mathcal{E}^1$ associated to $\phi$, so it suffices to show that $\{ \Phi_0 (s) \}_{s \ge 0}$ is contained in $(\mathcal{E}^{1 })^{\mathbb{K}}$.

By definition \cite[section 2.1.3]{Li20}, there exists a sequence $\{ \psi_j \}_j \subset (\mathcal{H}^{\mathrm{NA}})^{\mathbb{K}}$ which decreases to $\phi$. Since the non-Archimedean Monge--Amp\`ere energy is monotone along decreasing sequences by \cite[Theorem 6.9]{BJ18}, we find $\lim_{j} E^{\mathrm{NA}} (\psi_j) = E^{\mathrm{NA}} (\phi)$.

To each $\psi_j \in (\mathcal{H}^{\mathrm{NA}})^{\mathbb{K}}$ we can associate a subgeodesic ray of $\mathbb{K}$-invariant smooth K\"ahler potentials, and hence a (maximal) $C^{1 , \bar{1}}$-geodesic $\Psi_j$ by Phong--Sturm \cite[Theorem 1.1]{PS07}. We observe that $\Psi_j$ is $\mathbb{K}$-invariant, as follows. Indeed, in the notation of \cite[Theorem 1.1]{PS07}, the subgeodesic ray $\{ \phi (t;l) \}_{t \ge 0}$ is $\mathbb{K}$-invariant for all $l \in \mathbb{Z}_{>0}$, since the test configuration representing $\psi_j \in (\mathcal{H}^{\mathrm{NA}})^{\mathbb{K}}$ is $\mathbb{K}$-invariant and hence the ray $\phi (t;l)$ constructed as in \cite[section 4.2]{PS07} can be easily seen to be $\mathbb{K}$-invariant. The $C^{1 , \bar{1}}$-geodesic ray $\{ \Psi_j (t) \}_{t \ge 0}$ given by \cite[Theorem 1.1]{PS07} is defined by
\begin{equation*}
	\Psi_j (t) = \lim_{k \to \infty} \left( \sup_{l \ge k} [\phi(t;l)] \right)^* ,
\end{equation*}
where $\ast$ stands for the upper semicontinuous regularisation. It suffices to show that if $u (x,t)$ is a $\mathbb{K}$-invariant function on $X \times \mathbb{R}_{\ge 0}$ then so is $u^* (x,t)$; recall that the definition of $u^*$ is given by 
\begin{equation*}
	u^* (x,t) = \lim_{\epsilon \to 0} \sup_{d_{\omega}(x',x)+|t - t'| < \epsilon} u(x',t').
\end{equation*}
We then have, for any $k \in \mathbb{K}$,
\begin{align*}
	u^* (k \cdot x,t) &= \lim_{\epsilon \to 0} \sup_{d_{\omega} (x',k \cdot x)+|t - t'| < \epsilon} u(x',t') \\
	&=\lim_{\epsilon \to 0} \sup_{d_{\omega} (k^{-1} \cdot x', x)+|t - t'| < \epsilon} u(x',t') \\
	&=\lim_{\epsilon \to 0} \sup_{d_{\omega} (k^{-1} \cdot x', x)+|t - t'| < \epsilon} u(k^{-1} \cdot x',t') \\
	&=u^* (x,t)
\end{align*}
where in the second equality we used that $\mathbb{K}$ acts isometrically on $(X , \omega )$ (or the induced metric space $(X , d_{\omega})$) and in the third equality we used that $u$ is $\mathbb{K}$-invariant, showing that $u^*$ is indeed $\mathbb{K}$-invariant.

Now, any geodesic segment of $\Psi_j$ converges to the corresponding segment of the maximal geodesic $\Phi_0$ in $d_1$, by Lemma \ref{lmappgsg}. Since the $d_1$-limit of $\mathbb{K}$-invariant geodesic rays is also $\mathbb{K}$-invariant, we finally see that $\{ \Phi_0 (s) \}_{s}$ is contained in $(\mathcal{E}^{1 })^{\mathbb{K}}$.
\end{proof}

\begin{proof}[Proof of Lemma \ref{lmapenatc}]
Since $J_{\mathrm{ext}}$ is affine linear on geodesics, we may write $J_{\mathrm{ext}} (\Phi (s)) = c s $ for some constant $c \in \mathbb{R}$, as we normalised $J_{\mathrm{ext}}$ to be zero at $\Phi (0) = \Phi_{\mathrm{ref}}$. Thus we see that
	\begin{equation*}
		\lim_{s \to + \infty} \frac{J_{\mathrm{ext}} (\Psi (s))}{s} = J_{\mathrm{ext}} (\Psi (1))
	\end{equation*}
	for any geodesic ray $\Psi (s)$ in $( \mathcal{E}^1 )^{\mathbb{K}}$ emanating from $\Phi_{\mathrm{ref}}$. Thus the result follows from Lemmas \ref{lmappgsg} and \ref{lmmgdrkivt}, and the $d_1$-continuity of $J_{\mathrm{ext}}$.
\end{proof}

Following the above argument, we make the following definition.
\begin{definition} \label{dfjnaexsle1}
	We define a map $J_{\mathrm{ext}}^{\mathrm{NA}} : (\mathcal{E}^{1 , \mathrm{NA}})^{\mathbb{K}} \to \rl$ by
	\begin{equation*}
	J_{\mathrm{ext}}^{\mathrm{NA}} (\phi) := \lim_{s \to + \infty} \frac{J_{\mathrm{ext}} (\Phi (s))}{s} = J_{\mathrm{ext}} (\Phi (1))
	\end{equation*}
	where $\Phi$ is the unique maximal geodesic ray in $(\mathcal{E}^1)^{\mathbb{K}}$ associated to $\phi$, constructed by Berman--Boucksom--Jonsson \cite[Theorem 6.6]{BBJ} (see also Lemma \ref{lmmgdrkivt}).
\end{definition}

\begin{remark} \label{rmhleaegdf}
	Similar slope formulae were obtained by Han--Li \cite{HanLi23} for $g$-solitons. Apostolov--Jubert--Lahdili \cite{AJL23} and Lahdili \cite{Lah19} also proved ones for weighted cscK metrics with respect to K\"ahler test configurations (i.e.~ample test configurations with smooth total space and reduced central fibre); see also Inoue \cite{Ino} for $\mu$-cscK metrics.
\end{remark}

\begin{proposition} \label{ppctjestp}
Let $\{ \phi_j \}_j \subset (\mathcal{E}^{1 , \mathrm{NA}})^{\mathbb{K}}$ be a net that converges strongly to $\phi \in (\mathcal{E}^{1 , \mathrm{NA}})^{\mathbb{K}}$. Then 
\begin{equation*}
	\lim_{j} J_{\mathrm{ext}}^{\mathrm{NA}} (\phi_j) = J_{\mathrm{ext}}^{\mathrm{NA}} (\phi).
\end{equation*}

\end{proposition}

\begin{proof}
	If the net $\{ \phi_j \}_j \subset (\mathcal{E}^{1 , \mathrm{NA}})^{\mathbb{K}}$ decreases to $\phi$, it is immediate from Lemmas \ref{lmappgsg}, \ref{lmapenatc}, and Definition \ref{dfjnaexsle1} that we have indeed $J_{\mathrm{ext}}^{\mathrm{NA}} (\phi) = \lim_{j} J_{\mathrm{ext}}^{\mathrm{NA}} (\phi_j)$.
		
	When we take a net $\{ \phi_j \}_j \subset (\mathcal{E}^{1, \mathrm{NA}})^{\mathbb{K}}$ in general that converges strongly to $\phi$, we observe that there exists a decreasing net $\{ \psi_j \}_j \subset \mathcal{H}^{\mathrm{NA}} \subset \mathcal{E}^{1 , \mathrm{NA}}$ which converges to $\phi$ and $\psi_j \ge \phi_j$ for all $j$. This follows from the envelope property (see \cite[Theorem 5.20]{BJ18} and \cite[Corollary 3.16]{DXZ23}), which implies that the upper semicontinuous regularisation of $\sup_{j \ge l} \phi_j$, say $\psi^{(1)}_j$, is plurisubharmonic (on $X^{\mathrm{an}}$, in the sense introduced in \cite{BJ18}) for any $l$ and hence can be approximated from above by elements in $\mathcal{H}^{\mathrm{NA}}$. We further take a net in $\mathcal{H}^{\mathrm{NA}}$, say $\{ \psi^{(2)}_j \}_j$, which decreases to $\phi$ and define $\psi_j := \max \{ \psi^{(1)}_j, \psi^{(2)}_j \}$ (note that plurisubharmonic functions are stable under finite maxima \cite[page 649]{BJ18}, and we may assume that the index set is the same for both $\psi^{(1)}_j$ and $\psi^{(2)}_j$ by considering a product preorder as in \cite[Proof of Lemma 4.9]{BJ18}) and approximate it from above by an element of $\mathcal{H}^{\mathrm{NA}}$ if necessary.

	Since $\phi_j \to \phi$ strongly, we have $\lim_{j}  E^{\mathrm{NA}} (\phi_j) = E^{\mathrm{NA}} (\phi)$, and hence
	\begin{equation*}
		 \lim_{j} E^{\mathrm{NA}} (\phi_j) = E^{\mathrm{NA}} (\phi) =  \lim_{j} E^{\mathrm{NA}} (\psi_j)
	\end{equation*}
	since $\{ \psi_j \}_j$ decreases to $\phi \in (\mathcal{E}^{1 , \mathrm{NA}})^{\mathbb{K}}$. Now write $\{ \Phi_j (s) \}_{s \ge 0}$ and $\{ \Psi_j (s) \}_{s \ge 0}$ for the maximal geodesic rays associated to $\phi_j$ and $\psi_j$ respectively, which then implies $\lim_{s \to + \infty} E (\Phi_j (s) )/s = E^{\mathrm{NA}} (\phi_j)$ and $\lim_{s \to + \infty} E (\Psi_j (s) )/s = E^{\mathrm{NA}} (\psi_j)$. Since $\psi_j \ge \phi_j$, we find $\Psi_j \ge \Phi_j$ which further implies $E(\Psi_j (s)) - E (\Phi_j (s)) = d_1 (\Psi_j (s) , \Phi_j (s))$ by \cite[Proof of Corollary 4.14]{Darvas15}. Combining all these results, we conclude
	\begin{align*}
		0 &= \lim_{j} \lim_{s \to + \infty}\frac{1}{s} \left( E(\Psi_j (s)) - E (\Phi_j (s)) \right) \\
		&\ge \lim_{j} d_1 (\Psi_j (1) , \Phi_j (1)) \ge 0,
	\end{align*}
	just as we did in the proof of Lemma \ref{lmappgsg}. Noting that we have $d_1 (\Psi_j (1) , \Phi (1) ) \to 0$ by Lemma \ref{lmappgsg}, we finally find
	\begin{equation*}
		d_1 (\Phi_j (1) , \Phi (1)) \le d_1 (\Phi_j (1), \Psi_j (1)) + d_1 (\Psi_j (1) , \Phi (1)) \to 0,
	\end{equation*}
	which proves the claim by the $d_1$-continuity of $J_{\mathrm{ext}}$.
\end{proof}

The summary of the above argument is that the map $J_{\mathrm{ext}}^{\mathrm{NA}} : (\mathcal{H}^{\mathrm{NA}})^{\mathbb{K}} \to \rl$, defined in (\ref{eqyyaojnatc}), admits a continuous extension to $(\mathcal{E}^{1, \mathrm{NA}})^{\mathbb{K}}$.

\begin{corollary} \label{cctexjena}
	There exists a map $J_{\mathrm{ext}}^{\mathrm{NA}} : (\mathcal{E}^{1, \mathrm{NA}})^{\mathbb{K}} \to \rl$ which is continuous with respect to the strong topology, and agrees with Yao's formula
	\begin{equation*}
		J_{\mathrm{ext}}^{\mathrm{NA}} (\phi) = \frac{1}{V / n!} \frac{\mathcal{L}_{\beta}^{n+2}}{(n+2)!} - \frac{1}{(V / n!)^2} \left( \frac{\mathcal{L}^{n+1}}{(n+1)!} \right)^2
	\end{equation*}
	on $(\mathcal{H}^{\mathrm{NA}})^{\mathbb{K}}$.
\end{corollary}

Since we now have an invariant $J_{\mathrm{ext}}^{\mathrm{NA}} (\phi)$ defined for any $\phi \in (\mathcal{E}^{1 , \mathrm{NA}})^{\mathbb{K}}$, the generalisation of the relative uniform $K$-stability that we need for Theorem \ref{thexcl} can be given as follows.

\begin{definition} \label{dfguksrtm}
	$(X,L)$ is said to be \textbf{$\mathbb{G}$-uniformly $K$-stable relative to $\mathbb{T}_{\mathrm{ext}}$ over models} if there exists $\epsilon >0$ such that
	\begin{equation*}
		M^{\mathrm{NA}} (\phi) +  J_{\mathrm{ext}}^{\mathrm{NA}} (\phi) \ge \epsilon J^{\mathrm{NA}}_{\mathbb{T}} (\phi)
	\end{equation*}
	holds for any $\phi \in (\mathrm{PSH}^{\mathfrak{M} , \mathrm{NA}})^{\mathbb{K}}$.
\end{definition}

\section{Proof of Theorem \ref{thexcl}} \label{scpfomnth}

We are now ready to prove the main result, with the ingredients given so far. Suppose for contradiction that $(X,L)$ is $\mathbb{G}$-uniformly $K$-stable relative to $\mathbb{T}_{\mathrm{ext}}$ over $(\mathrm{PSH}^{\mathfrak{M} , \mathrm{NA}})^{\mathbb{K}}$ but does not admit an extremal metric. In this case, the modified Mabuchi energy fails to be coercive by \cite[Theorem 2]{He19}, which implies that there exists a finite energy geodesic ray $\{ \Phi (s) \}_{s \ge 0} \subset (\mathcal{E}^1)^{\mathbb{K}}$ such that
\begin{equation*}
	\lim_{s \to + \infty} \frac{M_{\mathrm{ext}} (\Phi (s))}{s} \le 0 \quad \text{and} \quad \inf_{\xi \in N_{\mathbb{R}}} \lim_{s \to + \infty} \frac{J (\Phi_{\xi} (s))}{s} = 1,
\end{equation*}
where $N_{\mathbb{R}} := \mathrm{Hom}_{\mathrm{Grp}} (\cx^*, \mathbb{T}) \otimes \mathbb{R}$; see \cite[Proof of Proposition 6.2]{Li20} for more details, and also \cite[Theorem 4.6]{NN25}. We further let $\phi \in (\mathcal{E}^{1 , \mathrm{NA}})^{\mathbb{K}}$ be the non-Archimedean metric associated to $\{ \Phi (s) \}_{s \ge 0}$.

First we find, by a result due to Li \cite[Theorem 1.7, Propositions 2.17 and 6.3]{Li20}, that there exists a sequence $\{ \phi_j \}_j \subset (\mathrm{PSH}^{\mathfrak{M} , \mathrm{NA}})^{\mathbb{K}}$ such that
\begin{equation*}
	\lim_{s \to + \infty} \frac{M (\Phi (s))}{s} \ge \lim_{j} M^{\mathrm{NA}} (\phi_j)
\end{equation*}
holds and that $\{ \phi_j \}_j$ converges to $\phi$ in the strong topology. Since $(X,L)$ is assumed to be $\mathbb{G}$-uniformly $K$-stable relative to $\mathbb{T}_{\mathrm{ext}}$ over models, there exists $\epsilon >0$ such that
\begin{equation*}
	M^{\mathrm{NA}} (\phi_j) +  J_{\mathrm{ext}}^{\mathrm{NA}} (\phi_j) \ge \epsilon J^{\mathrm{NA}}_{\mathbb{T}} (\phi_j)
\end{equation*}
holds for all $\phi_j \subset (\mathrm{PSH}^{\mathfrak{M} , \mathrm{NA}})^{\mathbb{K}}$.

Let $\Phi_j$ be the maximal geodesic ray in $(\mathcal{E}^{1})^{\mathbb{K}}$ associated to $\phi_j$ (Lemma \ref{lmmgdrkivt}). Thus we get, by arguing exactly as in \cite[Proof of Theorem 6.5]{Li20},
\begin{align*}
	\lim_{s \to + \infty}\frac{M (\Phi (s)) + J_{\mathrm{ext}} (\Phi (s))}{s}  &\ge \lim_{j} M^{\mathrm{NA}} (\phi_j) + \lim_{s \to + \infty} \frac{J_{\mathrm{ext}} (\Phi (s))}{s} \\
	&=\lim_{j} M^{\mathrm{NA}} (\phi_j) + \lim_{j} \lim_{s \to + \infty} \frac{J_{\mathrm{ext}} (\Phi_j (s))}{s} \\
	&=\lim_{j} \left( M^{\mathrm{NA}} (\phi_j) +  J_{\mathrm{ext}}^{\mathrm{NA}} (\phi_j) \right) \\
	&\ge \epsilon \lim_{j} J^{\mathrm{NA}}_{\mathbb{T}} (\phi_j) \\
	&= \epsilon \inf_{\xi \in N_{\mathbb{R}}} \lim_{s \to + \infty} \frac{J (\Phi_{\xi} (s))}{s}\\
	&=\epsilon >0
\end{align*}
which is a contradiction, where we used Proposition \ref{ppctjestp} and Definition \ref{dfjnaexsle1} in the second line, Definition \ref{dfjnaexsle1} in the third, and \cite[Corollary 6.1 and Lemma 6.4]{Li20} in the fifth.

We thus find that the modified Mabuchi energy is coercive on the space of $\mathbb{K}$-invariant K\"ahler potentials, and hence the extremal metric exists by \cite[Theorem 2]{He19}, if the manifold is $\mathbb{G}$-equivariantly uniformly $K$-stable relative to $\mathbb{T}_{\mathrm{ext}}$ over models, completing the proof of Theorem \ref{thexcl}.

\begin{remark}
If $X$ is of cohomogeneity one, by a result due to Odaka \cite[Appendix A]{Del20a}, we can even show that the sequence $\{ \phi_j \}_j $ in the above proof can be chosen from $ (\mathcal{H}^{\mathrm{NA}})^{\mathbb{K}}$. In this case the maximal geodesic is the $C^{1, \bar{1}}$-geodesic constructed by Phong--Sturm \cite[Theorem 1.1]{PS07}.	
\end{remark}

\section{Remark on $K$-polystability relative to the extremal torus} \label{scrlkpst}

Let $\mathbb{T}_{\mathrm{max}}$ be a maximal torus in $\mathrm{Aut}_0 (X,L)$, and $\mathbb{T}_{\mathrm{ext}} \subset \mathbb{T}_{\mathrm{max}}$ be the extremal torus. Stoppa--Sz\'ekelyhidi \cite{StoSze} proved that $(X,L)$ admitting an extremal metric is $K$-semistable relative to $\mathbb{T}_{\mathrm{ext}}$ with respect to $\mathbb{T}_{\mathrm{max}}$-equivariant test configurations. It is natural to expect that this result can be improved to the $K$-polystability relative to $\mathbb{T}_{\mathrm{ext}}$, again with respect to $\mathbb{T}_{\mathrm{max}}$-equivariant test configurations. Indeed, this result was proved by Mabuchi \cite{Mab14}, extending an earlier one due to Stoppa--Sz\'ekelyhidi \cite{StoSze} who proved the $K$-stability relative to $\mathbb{T}_{\mathrm{max}}$, but an alternative proof by using the variational principle following \cite{BDL20,He19} also seems interesting (see also \cite[Theorem 1.2]{Der18} and \cite[Theorem 2]{Lah19}). It is likely well-known to the experts but does not seem to be explicitly stated in the literature, and we briefly comment on how it can be proved by adapting \cite[section 4]{BDL20}.

We follow the argument and the notation of Berman--Darvas--Lu \cite[section 4]{BDL20}. Most of their argument can be immediately generalised to the extremal case; for example, \cite[Lemmas 4.2 and 4.3]{BDL20} hold true for extremal cases as well. Thus the only part that needs to be generalised to the extremal setting is \cite[Lemma 4.1]{BDL20}. We write $G := \mathrm{Aut}_0 (X, v_{\mathrm{ext}})$ for the identity component of the automorphism group which commutes with $v_{\mathrm{ext}}$, whose Lie algebra is known to be the sum of abelian Lie algebra and a reductive Lie algebra by the result of Calabi \cite[Theorem 1]{Cal2} (see also \cite[(2.5)]{He19}).

\begin{lemma}
	Suppose that $(X, \omega)$ is a K\"ahler manifold. Let $u_0 \in \mathcal{H}_0$ be an extremal potential and $\{ u_t \}_{t \ge 0} \subset \CE^1 \cap E^{-1} (0)$ be a finite energy geodesic ray emanating from $u_0$ such that
	\begin{equation} \label{lmeqifgbdd}
		\inf_{g \in G} J_{\omega} (g \cdot u_t) < C
	\end{equation}
	for some constant $C>0$ and uniformly for all $t \ge 0$. Then there exists $v \in \mathrm{Isom} (X, \omega_{u_0} )$ which is Hamiltonian and commutes with $v_{\mathrm{ext}}$ such that
	\begin{equation*}
		u_t = \exp (tJv) \cdot u_0.
	\end{equation*}
\end{lemma}

\begin{proof}
	Let $g_k \in G$ such that $J_{\omega} (g_k \cdot u_t) < C$. Then the theorem of Calabi \cite[Theorem 1]{Cal2} and the global Cartan decomposition (see e.g.~\cite[Proposition 6.2]{DR17}) implies that there exists $h_k \in \mathrm{Isom} (X, \omega_{u_0} , v_{\mathrm{ext}})$, an isometry commuting with the Killing vector field $v_{\mathrm{ext}}$ with respect to the extremal metric $\omega_{u_0}$, and $v_k \in \mathfrak{isom} (X, \omega_{u_0} , v_{\mathrm{ext}})$ such that $g_k = h_k \exp (-Jv_k)$. The rest of the argument is exactly as in \cite[Proof of Lemma 4.1]{BDL20}.
\end{proof}

The rest of the proof is exactly as in \cite[section 4]{BDL20}; note also that the right hand side of (\ref{lmeqifgbdd}) can be replaced by $C+ o(t)$, where $o(t)$ is a (non-negative) quantity satisfying $o(t)/t \to 0$ as $t \to + \infty$.

\bibliography{extremal.bib}

\end{document}